\documentclass{amsart}

\newtheorem{theorem}{Theorem}[section]
\newtheorem{thm}[theorem]{Theorem}

\newtheorem{lem}[theorem]{Lemma}

\theoremstyle{definition}

\newtheorem{prob}[theorem]{Problem}

\theoremstyle{remark}

\numberwithin{equation}{section}

\def\ocirc#1{\ifmmode\setbox0=\hbox{$#1$}\dimen0=\ht0
    \advance\dimen0 by1pt\rlap{\hbox to\wd0{\hss\raise\dimen0
    \hbox{\hskip.2em$\scriptscriptstyle\circ$}\hss}}#1\else
    {\accent"17 #1}\fi}

\DeclareMathOperator{\Gr}{Gr}
\DeclareMathOperator{\Mon}{Mon}
\DeclareMathOperator{\Cl}{Cl}

\newcommand{\nin}{\notin}

\renewcommand{\L}{\mathfrak L}

\renewcommand{\AA}{\mathcal A}
\newcommand{\K}{\mathfrak K}

\newcommand{\M}{\mathfrak M}

\newcommand{\mult}{\times}

\newcommand{\ini}{\mathbin{\triangleleft}}

\begin{document}

\title[Universality of the lattice of transformation monoids]{Universality of the lattice of transformation monoids}

\author{Michael Pinsker}
    \address{\'{E}quipe de Logique Math\'{e}matique\\ Universit\'{e} Denis Diderot -- Paris 7\\
	UFR de Math\'{e}matiques\\
	75205 Paris Cedex 13, France}
    \email{marula@gmx.at}
    \urladdr{http://dmg.tuwien.ac.at/pinsker/}
    \thanks{Research of the first author supported by an APART-fellowship of the Austrian Academy of Sciences.}

\author{Saharon Shelah}
    \address{Institute of Mathematics\\ The Hebrew University of Jerusalem\\91904 Jerusalem, Israel, and Department of Mathematics, Rutgers University, New Brunswick, New Jersey 08854}
    \email{shelah@math.huji.ac.il}
    \urladdr{http://shelah.logic.at}
\thanks{Research of the second author supported by
  German-Israeli Foundation for Scientific Research \& Development
  Grant No. 963-98.6/2007. Publication 983 on Shelah's list.}
  \thanks{The authors would like to thank an anonymous referee for his valuable comments which led to significant improvements in the presentation of the paper.}



\commby{Julia Knight}

\begin{abstract}
 The set of all transformation monoids on a fixed set of infinite cardinality $\lambda$, equipped with the order of inclusion, forms a complete algebraic lattice $\Mon(\lambda)$ with $2^\lambda$ compact elements. We show that this lattice is universal with respect to closed sublattices, i.e., the closed sublattices of $\Mon(\lambda)$ are, up to isomorphism, precisely the complete algebraic lattices with at most $2^\lambda$ compact elements.
\end{abstract}

\maketitle

\section{Definitions and the result}

Fix an infinite set -- for the sake of simpler notation, we identify the set with its cardinality $\lambda$. By a \emph{transformation monoid} on $\lambda$ we mean a subset of $\lambda^\lambda$ which is closed under composition and which contains the identity function. The set of transformation monoids acting on $\lambda$, ordered by inclusion, forms a complete lattice $\Mon(\lambda)$, in which the meet of a set of monoids is simply their intersection. This lattice is \emph{algebraic}, i.e., every element is a join of compact elements -- an element $a$ in a complete lattice $\L=(L,\vee,\wedge)$ is called \emph{compact} iff whenever $A\subseteq L$ and $a\leq\bigvee A$, then there is a finite $A'\subseteq A$ such that $a\leq\bigvee A'$. In the case of $\Mon(\lambda)$, the compact elements are precisely the finitely generated monoids, i.e., those monoids which contain a finite set of functions such that every function of the monoid can be composed from functions of this finite set. 
Consequently, the number of compact elements of $\Mon(\lambda)$ equals $2^\lambda$.

It is well-known and not hard to see that for any cardinal $\kappa$, the algebraic lattices with at most $\kappa$ compact elements are, up to isomorphism, precisely the subalgebra lattices of algebras whose domains have $\kappa$ elements (\cite{BirkhoffFrink}; see also Theorem~48 in the textbook~\cite{Graetzer}). For example, $\Mon(\lambda)$ is the subalgebra lattice of the algebra which has domain $\lambda^\lambda$, a binary operation which is the function composition on $\lambda^\lambda$, as well as a constant operation whose value is the identity function on $\lambda$.

Let $\K$ and $\L$ be complete lattices such that the domain of $\L$ is contained in the domain of $\K$. Then $\L$ is called a \emph{complete sublattice} of $\K$ iff all joins and meets in $\L$ equal the corresponding joins and meets in $\K$; in most of the literature where the notion is mentioned, and in particular in the widely used textbook~\cite{Graetzer}, this includes empty joins and meets, and so in this definition the respective largest and smallest elements of $\L$ and $\K$ must coincide. When the condition of ``complete sublattice'' holds for all \emph{non-empty} joins and meets, then we say that $\L$ is a \emph{closed sublattice} of $\K$; the difference here is that the respective largest and smallest elements of $\L$ and $\K$ need not coincide. In other words, $\L$ is a closed sublattice of $\K$ if and only if it is a complete sublattice of a closed interval of $\K$.

Now let $\K$ be a complete algebraic lattice and $\L$ be a closed sublattice of $\K$. Then it is a folklore fact that $\L$ is algebraic as well, and that the number of compact elements of $\L$ equals at most the corresponding number for $\K$. For the comfort of the reader, let us sketch the argument showing this. Denote for every compact element $x$ of $\K$ which is below some element of $\L$  the smallest element of $\L$ which is above $x$ by $x^\L$. Then it is not hard to see that the compact elements of $\L$ are precisely the elements of the form $x^\L$. Hence, the mapping that sends every $x$ as above to $x^\L$ shows that $\L$ does not possess more compact elements than $\K$, and it also follows easily from the above that $\L$ is algebraic.

By this observation, any closed sublattice of $\Mon(\lambda)$ is algebraic and has at most $2^\lambda$ compact elements. In this paper, we prove the converse of this fact. This had been stated as an open problem in~\cite[Problem C]{GoldsternPinsker08}. We remark that it is clear from the context in~\cite{GoldsternPinsker08} that the word ``subinterval'' in the formulation of Problem~C is an error; it is Problem~B which asks about subintervals. We also note that while in~\cite{GoldsternPinsker08} the authors write ``complete sublattice'', they confirmed upon inquiry that they really meant ``closed sublattice'', although they consider the question about complete sublattices interesting as well.

\begin{thm}\label{thm:main}
    $\Mon(\lambda)$ is \emph{universal} for complete algebraic lattices with at most $2^\lambda$ compact elements with respect to closed sublattices, i.e., the closed sublattices of $\Mon(\lambda)$ are, up to isomorphism, precisely the complete algebraic lattices with at most $2^\lambda$ compact elements.
\end{thm}

We remark that it follows from our proof that if $\L$ is an algebraic lattice with at most $2^\lambda$ compact elements, then it is even isomorphic to a closed sublattice of $\Mon(\lambda)$ via an isomorphism which preserves the smallest element (but not the largest, in which case we would obtain a complete sublattice).

\section{Related work and possible extensions}

\subsection{Cardinality questions and non-closed sublattices}

It has been known for a long time that every (not necessarily complete) lattice $\L$ is isomorphic to a sublattice of the lattice of subgroups of a group~\cite{Whitman}. Hence, viewing the group as a monoid, it follows that every lattice is isomorphic to a sublattice of the lattice of submonoids of a monoid $\M$. Strengthenings of the latter statement were obtained in~\cite{Repnitskij}, where it was shown that one can impose a variety of different additional properties on the monoid $\M$. In these theorems, one cannot simply replace ``sublattice'' by ``closed sublattice'', since the lattice of subalgebras of an algebra is algebraic, and hence, by our discussion above, closed sublattices must share the same property.

There are lattices $\L$ of size $\kappa$ for which the corresponding monoid $\M$ must have size at least $\kappa$ as well: for example, it is easy to see that this is the case when $\L$ is the lattice on $\kappa$ induced by the natural order of $\kappa$. When $\L$ is complete, algebraic, and has $\kappa$ compact elements, then the monoid $\M$ will generally have to have size at least $\kappa$, although $\L$ itself might very well have size $2^\kappa$; this improvement on the obvious general lower bound for the cardinality of $\M$ compared to arbitrary lattices is due to the fact that the structure of $\L$ is already determined by the structure of the join-semilattice of its compact elements. A given abstract monoid $\M$ can then in turn be realized as a transformation monoid by letting it act on itself; hence, it will act on a set of size $\kappa$. In our case, this would yield an embedding into $\Mon(2^\lambda)$; the difficulty of our theorem is to find $\L$ as a closed sublattice of the optimal $\Mon(\lambda)$.

\subsection{Closed sublattices of related algebraic lattices}

A \emph{clone} on $\lambda$ is a set of finitary operations on $\lambda$ which is closed under composition and which contains all finitary projections; in other words, it is a set of finitary operations closed under building of terms (without constants). The set of all clones on $\lambda$, ordered by inclusion, also forms a complete algebraic lattice $\Cl(\lambda)$ with $2^\lambda$ compact elements, into which $\Mon(\lambda)$ embeds naturally, since a transformation monoid can be viewed as a clone all of whose operations depend on at most one variable. Universality of $\Cl(\lambda)$ for complete algebraic lattices with at most $2^\lambda$ compact elements with respect to closed sublattices has been shown in~\cite{Pin06AlgebraicSublattices} (the author of~\cite{Pin06AlgebraicSublattices} writes ``complete sublattices'' but really means -- and proves -- ``closed sublattices''); our result is a strengthening of this result.

Observe that similarly to transformation monoids and clones, the set of \emph{permutation groups} on $\lambda$ forms a complete algebraic lattice $\Gr(\lambda)$ with respect to inclusion. By virtue of the identity embedding, $\Gr(\lambda)$ is a complete sublattice of $\Mon(\lambda)$. We do not know the following.

\begin{prob} \label{prob:groups}
 Is every complete algebraic lattice with at most $2^\lambda$ compact elements a closed sublattice of $\Gr(\lambda)$?
\end{prob}

In this context it is worthwhile mentioning that in our proof of Theorem~\ref{thm:main}, we exclusively use monoids which only contain permutations. In other words, we construct for every complete algebraic lattice with at most $2^\lambda$ compact elements a closed sublattice of the interval of $\Mon(\lambda)$ consisting of those monoids on $\lambda$ which are subsets of the symmetric group on $\lambda$. However, our ``permutation monoids" are themselves no groups, and in fact they never contain the inverse of any of their permutations (except the identity) -- adding inverses would collapse the construction.

A related problem is which lattices appear as \emph{intervals} of $\Gr(\lambda)$, $\Mon(\lambda)$, and $\Cl(\lambda)$. This remains open -- for the latter two lattices this question has been posed as an open problem in~\cite{GoldsternPinsker08} (Problems~B and~A, respectively). By a deep theorem due to T\ocirc{u}ma~\cite{Tum89subgroupLattices}, every complete algebraic lattice with $\lambda$ compact elements is isomorphic to an interval of the subgroup lattice of a group of size $\lambda$; from this it only follows that $\Gr(\lambda)$ contains all complete algebraic lattices with at most $\lambda$ compact elements as intervals. Proving that $\Gr(\lambda)$ contains all complete algebraic lattices with at most $2^\lambda$ compact elements as intervals would be a common strengthening of T\ocirc{u}ma's result and a positive answer to Problem~\ref{prob:groups}.

\section{Proof of the theorem}

\subsection{Independent composition engines}

For a cardinal $\kappa$ and a natural number $n\geq 1$, we write $\Lambda^n_\kappa:=\kappa^n\mult2^n$. We set $\Lambda_\kappa:=\bigcup_{n\geq 1}\Lambda^n_\kappa$. For sequences $p,q$, we write $p\ini q$ if $p$ is a non-empty initial segment of $q$ (we consider $q$ to be an initial segment of itself). For $(\eta,\phi)$ and $(\eta',\phi')$ in $\Lambda_\kappa$, we also write $(\eta,\phi)\ini (\eta',\phi')$ if $\eta\ini\eta'$ and $\phi\ini\phi'$. If $p$ is a sequence and $r$ a set, then $p\ast r$ denotes the extension of $p$ by the element $r$. We write $\langle{r}\rangle$ for the one-element sequence containing only $r$.

A sequence $P$ of elements of $\Lambda_\kappa$ is \emph{reduced} iff it does not contain both $(\eta\ast\alpha,\ \phi\ast 0)$ and $(\eta\ast\alpha,\ \phi\ast 1)$ for any $(\eta,\phi)\in\Lambda_\kappa$ and $\alpha\in\kappa$. We call two sequences $P, Q$ \emph{equivalent} iff $P$ can be transformed into $Q$ by permuting its elements.

For a set $W$ and a cardinal $\kappa$, a \emph{$\kappa$-branching independent composition engine ($\kappa$-ICE) on $W$} is an indexed set $\{f_{(\eta, \phi)}: (\eta,\phi)\in \Lambda_\kappa\}$ of permutations on $W$ satisfying all of the following:
\begin{itemize}
    \item[(i)] (Composition) For all $(\eta,\phi)\in\Lambda_\kappa$ and for all $\alpha\in\kappa$ we have $f_{(\eta, \phi)}=f_{(\eta\ast\alpha,\ \phi\ast 0)}\circ f_{(\eta\ast\alpha,\ \phi\ast 1)}$;
    \item[(ii)] (Commutativity) For all $a,b \in\Lambda_\kappa$ we have that $f_{a}\circ f_{b}=f_{b}\circ f_{a}$.
    \item[(iii)] (Independence) Whenever $P=(p_1,\ldots,p_n), Q=(q_1,\ldots,q_m)\subseteq\Lambda_\kappa$ are inequivalent  reduced sequences, then $t_P:=f_{p_1}\circ\cdots\circ f_{p_n}$ and $t_Q:=f_{q_1}\circ\cdots\circ f_{q_m}$ are not equal.
\end{itemize}

Note that by the commutativity of the system, the order of the elements of the sequences $P$ and $Q$ in condition (iii) is not of importance.

\begin{lem}
    There exists a $2^\lambda$-ICE on $\lambda$.
\end{lem}
\begin{proof}
    We show that there exists a $2^\lambda$-ICE on $W:=\lambda\mult \mathbb{Z}$. Let 
    $$
    \AA:=\{A_{(\eta,\phi)}: (\eta,\phi)\in \Lambda_{2^\lambda} \text{ and the last entry of } \phi \text{ equals } 0\}
    $$
    be an independent family of subsets of $\lambda$, i.e., any non-trivial finite Boolean combination of these sets is non-empty (see, for example,~\cite[Lemma~7.7]{Jec03} for a proof of the existence of such a family). For all $(\eta,\phi)\in \Lambda_{2^\lambda}$, set $\#A_{(\eta,\phi)}$ to equal $A_{(\eta,\phi)}$, if the last entry of $\phi$ equals $0$, and $\lambda\setminus A_{(\eta,\phi')}$ otherwise, where $\phi'$ is obtained from $\phi$ by changing the last entry to $0$. Now define $B_{(\eta,\phi)}:=\bigcap_{s\ini (\eta,\phi)} \#A_s$, for all $(\eta,\phi)\in \Lambda_{2^\lambda}$. 
    
    We will define the $2^\lambda$-ICE by means of the family $\{B_{(\eta,\phi)}: (\eta,\phi)\in \Lambda_{2^\lambda}\}$ as follows. For all $(\eta,\phi)\in\Lambda_{2^\lambda}$ and all $(\alpha,i)\in W$, we set
    $$
        f_{(\eta,\phi)}(\alpha,i)=
            \begin{cases}
                (\alpha, i+1) &, \text{ if } \alpha\in B_{(\eta,\phi)},\\
                (\alpha,i) &, \text{ otherwise.}
            \end{cases}
    $$
   
  We claim that this defines a $2^\lambda$-ICE on $W$. Clearly,~(ii) of the definition is satisfied. Property~(i) is a direct consequence of the fact that for all $(\eta,\phi)\in \Lambda_{2^\lambda}$ and all $\alpha< \lambda$, $B_{(\eta,\phi)}$ is the disjoint union of $B_{(\eta\ast\alpha,\phi\ast 0)}$ and $B_{(\eta\ast\alpha,\phi\ast 1)}$. 
  
  To see~(iii), let $P$ and $Q$ be reduced and inequivalent. We first claim that we can assume that $P$ and $Q$ have no entries in common. So say that  some $(\eta,\phi)$ of $\Lambda_{2^\lambda}$ occurs in both $P$ and $Q$. Then, since $t_P=t_Q$ if and only if $f_{(\eta,\phi)}^{-1}\circ t_P=f_{(\eta,\phi)}^{-1}\circ t_Q$, proving that $t_P$ is not equal to $t_Q$ is the same as proving that $t_{P'}$ is not equal to $t_{Q'}$, where for $X\in\{P,Q\}$ we write  $X'$ for the sequence obtained from $X$ by removing one occurrence of $(\eta,\phi)$. Observe that as subsequences of $P$ and $Q$ respectively, $P'$ and $Q'$ are still reduced. Repeating this process, we may indeed assume that $P$ and $Q$ have no common entries. Now to prove that $t_P\neq t_Q$, let $(\eta,\phi)$ be an element of $\Lambda_{2^\lambda}$ which appears either in $P$ or in $Q$ and which has the property that $(\eta',\phi')= (\eta,\phi)$ for all pairs $(\eta',\phi')$ which appear in $P$ or $Q$ and for which $(\eta',\phi')\ini (\eta,\phi)$. Say without loss of generality that $(\eta,\phi)$ appears in $P$. Let $A$ be the union of all $B_{q}$ for which $q$ appears in $Q$. Then it follows from the independence of the family $\AA$, from the fact that $Q$ is reduced, and from the fact that $Q$ contains no $q$ with $q\ini (\eta,\phi)$ that $B_{(\eta,\phi)}$ is not contained in the union of all $\#A_q$
for $q$ appearing in $Q$, and thus 
  $B_{(\eta,\phi)}\setminus A$ is non-empty. Let $\alpha$ be an element of the latter set. Then $t_P(\alpha,0)=(\alpha,k)$ for some $k >0$, whereas $t_Q(\alpha,0)=(\alpha,0)$. Hence, $t_P\neq t_Q$.
\end{proof}

\subsection{From lattices to monoids}

For a $\kappa$-ICE $\{f_{(\eta,\ \phi)}: (\eta,\phi)\in \Lambda_\kappa\}$ on $W$ and a subset $S$ of $\Lambda_\kappa$, we set $F(S)$ to be the monoid \emph{generated} by the functions with index in $S$, i.e., the smallest monoid of functions from $W$ to $W$ which contains all the functions with index in $S$. Another way to put it is that $F(S)$ contains precisely the composites of functions with index in $S$ as well as the identity function on $W$.

In the following, fix a $2^\lambda$-ICE $\{f_{(\eta,\ \phi)}: (\eta,\phi)\in \Lambda_{2^\lambda}\}$ on $\lambda$.
Let $\L=(L,\vee,\wedge)$ be any complete algebraic lattice with $2^\lambda$ compact elements. Let $C\subseteq L$ be the set of compact elements of $\L$ excluding the smallest element. Enumerate $C$, possibly with repetitions, by $\{c_{(\eta,\phi)}:(\eta,\phi)\in \Lambda_{2^\lambda}\}$, and in such a way that the following hold:
\begin{itemize}
    \item[(1)] Every element of $C$ is equal to $c_{(\langle{\alpha}\rangle,\langle{0}\rangle)}$ for some $\alpha<2^\lambda$;
    \item[(2)] For all $(\eta,\phi)\in \Lambda_{2^\lambda}$ and all $\alpha\in 2^\lambda$ we have $c_{(\eta,\phi)}\leq c_{(\eta\ast\alpha,\phi\ast 0)}\vee c_{(\eta\ast\alpha,\phi\ast 1)}$;
    \item[(3)] For all $(\eta,\phi)\in \Lambda_{2^\lambda}$ and all $d,d'\in C$ with $c_{(\eta,\phi)}\leq d\vee d'$, there exists $\alpha<2^\lambda$ such that $d=c_{(\eta\ast\alpha,\phi\ast 0)}$ and $d'=c_{(\eta\ast\alpha,\phi\ast 1)}$.
\end{itemize}

Consider the semilattice $(C,\vee)$ of compact elements of $\L$ without the smallest element. An \emph{ideal} of $(C,\vee)$ is a possibly empty subset of $C$ which is downward closed and closed under finite joins. It is a straightforward consequence of the fact that $\L$ is algebraic that $\L$ is isomorphic to the lattice of ideals of $(C,\vee)$ (the statement is equivalent to Theorem~42 in~\cite{Graetzer}). The meet $\bigwedge_{u\in U} I_u$ of a non-empty set of ideals $\{I_u:u\in U\}$ in this lattice is just their intersection; their join $\bigvee_{u\in U} I_u$ the smallest ideal containing all $I_u$, that is, the set of all elements $c$ of $C$ for which there exist $c_1,\ldots,c_n\in \bigcup_{u\in U} I_u$ such that $c\leq c_1\vee\cdots\vee c_n$.

To every ideal $I\subseteq C$, assign the sets $S(I):=\{(\eta,\phi)\in \Lambda_{2^\lambda}: c_{(\eta,\phi)}\in I\}$, and $F(I):=F(S(I))$.

\begin{lem}
    If $\{I_u:u\in U\}$ is a non-empty set of ideals of $(C,\vee)$, then $\bigvee_{u\in U} F(I_u)=F(\bigvee_{u\in U} I_u)$.
\end{lem}
\begin{proof}
    The inclusion $\subseteq$ is trivial. For the other direction, it is enough to show that if $c_{(\eta,\phi)}$ is an element of $\bigvee_{u\in U} I_u$, then $f_{(\eta,\phi)}$ is an element of $\bigvee_{u\in U} F(I_u)$. There exist $c_{(\eta_1,\phi_1)},\ldots,c_{(\eta_n,\phi_n)}\in\bigcup_{u\in U} I_u$ such that $c_{(\eta,\phi)}\leq c_{(\eta_1,\phi_1)}\vee\cdots\vee c_{(\eta_n,\phi_n)}$. We use induction over $n$. If $n=1$, then $c_{(\eta,\phi)}\leq c_{(\eta_1,\phi_1)}\in I_u$ for some $u\in U$, so $c_{(\eta,\phi)}\in I_u$. Hence, $f_{(\eta,\phi)}\in F(I_u)$, and we are done. In the induction step, suppose the claim holds for all $1\leq k<n$. Set $d:=c_{(\eta_1,\phi_1)}\vee\cdots\vee c_{(\eta_{n-1},\phi_{n-1})}$ and $d':=c_{(\eta_n,\phi_n)}$. Since $c_{(\eta,\phi)}\leq d\vee d'$, there exist $\alpha<2^\lambda$ such that $(d,d')=(c_{(\eta\ast\alpha,\phi\ast 0)},c_{(\eta\ast\alpha,\phi\ast 1)})$, by Property~(3) of our enumeration. 
   By the induction hypothesis, we have $f_d, f_{d'}\in\bigvee_{u\in U} F(I_u)$. Since $f_{(\eta,\phi)}=f_{(\eta\ast\alpha,\phi\ast 0)}\circ f_{(\eta\ast\alpha,\phi\ast 1)}$, we get that $f_{(\eta,\phi)}\in \bigvee_{u\in U} F(I_u)$ as well, proving the lemma. 
\end{proof}

\begin{lem}
    If $\{I_u:u\in U\}$ is a non-empty set of ideals of $(C,\vee)$, then $\bigcap_{u\in U} F(I_u)=F(\bigcap_{u\in U} I_u)$.
\end{lem}
\begin{proof}
	This time, the inclusion $\supseteq$ is trivial. For the other direction, let $t\in \bigcap_{u\in U} F(I_u)$, and assume that $t$ is not the identity function. Then there is a unique reduced set $P$ such that $t=t_P$, by Property~(iii) of an independent composition engine. Now let $u\in U$ be arbitrary. Then there exists a sequence $Q$ in $S(I_u)$ such that $t=t_Q$. By subsequently replacing two entries $(\eta\ast\alpha,\phi\ast 0), (\eta\ast\alpha,\phi\ast 1)$ in $Q$ by $(\eta,\phi)$, we obtain a reduced sequence $Q'$ which still satisfies $t=t_{Q'}$. Since $I_u$ is closed under joins, and by Property~(2) of our enumeration of the compact elements, all entries of $Q'$ are still elements of $S(I_u)$. By the independence property, $P$ and $Q'$ are equivalent, and hence $P$ is a sequence in $S(I_u)$. But $u$ was arbitrary, so $P$ is a sequence in $\bigcap_{u\in U} S(I_u)$. This proves $t\in F(\bigcap_{u\in U} I_u)$.
\end{proof}

\begin{lem}
	If $I, J$ are distinct ideals of $(C,\vee)$, then $F(I)\neq F(J)$.
\end{lem}
\begin{proof}
	Assume without loss of generality that $I\setminus J$ is non-empty. Then $S(I)\setminus S(J)$ is non-empty as well; let $(\eta,\phi)$ be an element therein. If we had $f_{(\eta,\phi)}\in F(J)$, then there would exist a sequence $Q$ in $S(J)$ such that $f_{(\eta,\phi)}=t_Q$. As in the proof of the preceding lemma, we could moreover assume that $Q$ is reduced. But then by the independence property, the sequence $Q$ would be equivalent with the sequence $\langle{(\eta,\phi)}\rangle$, implying $(\eta,\phi)\in S(J)$ -- a contradiction. So $f_{(\eta,\phi)}\nin F(J)$ but $f_{(\eta,\phi)}\in F(I)$, and we are done.
\end{proof}

It follows from the above that the assignment $I\mapsto F(I)$ is an injective mapping from the lattice of ideals of $(C,\vee)$ to $\Mon(\lambda)$ which preserves arbitrary non-empty joins and meets, proving our theorem. Moreover, under this mapping the smallest element of $\L$, represented by the empty ideal in $(C,\vee)$, is sent to the monoid which contains only the identity function, and hence to the smallest element of $\Mon(\lambda)$. This justifies the remark after the theorem.

\providecommand{\MRhref}[2]{%
  \href{http://www.ams.org/mathscinet-getitem?mr=#1}{#2}
}
\providecommand{\href}[2]{#2}

\end{document}